\documentclass{amsart}
\usepackage{amsmath}
\usepackage{graphicx} 
\usepackage{amsfonts} 
\usepackage{amssymb}
\usepackage{color}
\usepackage{mathrsfs}
\usepackage{epsfig}
\usepackage{url}
\usepackage{mathrsfs,a4wide}
\usepackage{calligra}
\usepackage{pgf,tikz,pgfplots} 
\theoremstyle{plain}
\newtheorem{theorem}{Theorem}[section]

\newtheorem{lemma}[theorem]{Lemma}

\theoremstyle{definition}
\newtheorem{remark}[theorem]{Remark}
\theoremstyle{remark}
\numberwithin{equation}{section}

\newcommand{\sym}{\mathrm{sym}}
\newcommand{\Tl}{\mathbb{T}}
\newcommand{\T}{\mathcal{T}}

\newcommand{\D}{\mathrm{D}}

\newcommand{\ep}{\varepsilon}
\newcommand{\e}{\varepsilon}

\newcommand{\ffi}{\varphi}

\newcommand{\Span}{\mathrm{span}\,}

\newcommand{\C}{\mathbb{C}}
\newcommand{\R}{\mathbb{R}}
\newcommand{\N}{\mathbb{N}}
\newcommand{\Z}{\mathbb{Z}}

\newcommand{\E}{\mathcal{E}}
\newcommand{\F}{\mathcal{F}}
\newcommand{\AS}{\mathcal{S}}
\newcommand{\AD}{\mathcal{D}}

\newcommand{\di}{\textrm{dist}}

\newcommand{\ud}{\mathrm{d}}

\newcommand{\Om}{\Omega}
\newcommand{\supp}{\mathrm{supp}\,}
\newcommand{\M}{\mathcal{M}}

\newcommand{\as}{\mathscr{A}\!\mathscr{S}}
\newcommand{\Huno}{\mathcal H^1}

\newcommand{\weakly}{\rightharpoonup}           
\newcommand{\weakstar}{\stackrel{*}{\weakly}}   
\newcommand{\fla}{\stackrel{\mathrm{flat}}{\rightarrow}}
\newcommand{\flt}{\mathrm{flat}}

\newcommand{\conv}{\mathrm{conv}}

\newcommand{\skw}{{\mathrm{skew}}}
\newcommand{\DT}{\mathbb{T}'_{+}}

\def\XXint#1#2#3{{\setbox0=\hbox{$#1{#2#3}{\int}$}
     \vcenter{\hbox{$#2#3$}}\kern-.5\wd0}}

\usepackage{latexsym}
\newcommand{\newatop}{\genfrac{}{}{0pt}{1}} 

\makeatletter
\def\@splitop#1#2\@nil{$\mathscr{#1}\!\!$\calligra#2\,\,}
\newcommand*\DeclareCursiveOperator[2]{%

 \newcommand#1{\mathop{\mbox{\@splitop#2\@nil}}\nolimits}}
\makeatother
\DeclareCursiveOperator{\Anew}{A}
\DeclareCursiveOperator{\Bnew}{B}
\DeclareCursiveOperator{\Cnew}{C}
\DeclareCursiveOperator{\Dnew}{D}
\DeclareCursiveOperator{\Enew}{E}
\DeclareCursiveOperator{\Qnew}{Q}

\usepackage{hyperref}

\title[Coarse-graining for edge dislocations] {Coarse-graining of a discrete model for edge dislocations in the regular triangular lattice}

\thanks{\today}

\author[R. Alicandro]
{R. Alicandro}
\address[Roberto Alicandro]{DIEI, Universit\`a di Cassino e del Lazio meridionale, via Di Biasio 43, 03043 Cassino (FR), Italy}
\email[R. Alicandro]{alicandr@unicas.it}
\author[L. De Luca]
{L. De Luca}
\address[Lucia De Luca]{Istituto per le Applicazioni del Calcolo ``M. Picone'', IAC-CNR, via dei Taurini 19, 00185 Roma, Italy}
\email[L. De Luca]{lucia.deluca@cnr.it}
\author[G. Lazzaroni]
{G. Lazzaroni}
\address[Giuliano Lazzaroni]{DiMaI, Universit\`a di Firenze, Viale Morgagni 67/a, 50134 Firenze, Italy}
\email[G. Lazzaroni]{giuliano.lazzaroni@unifi.it}
\author[M. Palombaro]
{M. Palombaro}
\address[Mariapia Palombaro]{DISIM, Universit\`a dell'Aquila, Via Vetoio, 67100 L'Aquila, Italy}
\email[M. Palombaro]{mariapia.palombaro@univaq.it}
\author[M. Ponsiglione]
{M. Ponsiglione}
\address[Marcello Ponsiglione]{Dip. di Matematica, Univ. Roma-I ``La Sapienza'', Piazzale Aldo Moro 5, 00185 Roma, Italy}
\email[M. Ponsiglione]{ponsigli@mat.uniroma1.it}

\begin{document}

\begin{abstract}
We consider a discrete model of planar elasticity where the particles, in the reference configuration,  sit on a regular triangular lattice and interact through nearest neighbor pairwise potentials, with bonds modeled as 
linearized elastic springs. 
Within this framework we introduce plastic slip fields, 
whose discrete circulation around each triangle detects the possible presence of an edge dislocation.

We provide  
a $\Gamma$-convergence analysis, as the lattice spacing tends to zero, of the elastic energy induced by  edge dislocations in the energy regime corresponding to a finite number of geometrically necessary dislocations. 

\end{abstract}
\maketitle

\tableofcontents

\section*{Introduction}
Dislocations are  line defects in the periodic structure of crystals and are considered the main microscopic mechanism of plastic flow.
Idealized straight dislocations are  classified into two  types, {\it edge} and {\it screw}, while in real crystals dislocations are actually curved lines of mixed type  \cite{HB,HL}. 
From a mathematical viewpoint, dislocations can be seen as  
topological line singularities around which the elastic strain 
has non trivial circulation given by a vector of the underlying lattice and referred to as {\it Burgers vector}.

Here we focus on  planar elasticity, where  the relevant dislocations are of  edge nature and  can be seen as 
topological point singularities  of the elastic strain.
Specifically, we study the discrete elastic energy induced by a finite system of edge dislocations in a finite portion $\Omega\cap\ep\Tl$ of the regular triangular lattice $\e \Tl$\,, where $\ep$ is the lattice spacing. 

As customary in the linearized framework, we adopt the additive decomposition of the {\it discrete deformation gradient} $\ud u$\,, defined on pairs of nearest neighbors, into an elastic and a plastic part.
The latter is represented by an additional variable $\sigma$\,, referred to as {\it slip}, defined on pairs of nearest neighbors and taking values in the  set of lattice vectors. In this way  we identify dislocations as points around which the discrete circulation of the plastic slip $\sigma$, or, equivalently, the discrete circulation of $\ud u-\sigma$\,, is non trivial. Such a procedure agrees with the formalism of the eigenstrains considered in \cite{AO}. 

We focus on the simple case where 
nearest neighbors interact through linearized pairwise potentials;
specifically, in absence of singularities, the energy functional is defined by 
\begin{equation*}
G_\ep(u):=\frac{1}{2\ep^2}\sum_{\newatop{i,j\in\Omega\cap\ep\Tl}{|i-j|=\ep}}\big(\ud u(i,j)\cdot(j-i)\big)^2\,.
\end{equation*}
The above energy can be formally derived by linearization of suitable nonlinear frame invariant functionals (see Remark \ref{linearizzazione}); further rigorous linearization results  in terms of $\Gamma$-convergence are provided in  
  \cite{ALP, BSV, Schm09}. In the limit as $\e\to 0$ 
the energy $\frac{1}{\ep^2}G_\e$ approximates,
 up to a pre-factor, 
the continuum isotropic elastic energy with Lam\'e parameters both equal to $1$ (see Remark \ref{disccont}).

In presence of a plastic slip field $\sigma$, the energy functional should depend only on the elastic strain $\ud u-\sigma$\,, and hence the resulting elastic energy reads as
\begin{equation*}
F_\ep(u,\sigma):=\frac{1}{2\ep^2}\sum_{\newatop{i,j\in\Omega\cap\ep\Tl}{|i-j|=\ep}}\big((\ud u(i,j)-\sigma(i,j))\cdot(j-i)\big)^2\,.
\end{equation*} 
It is well known that the energy induced by an isolated edge dislocation, both in the discrete and in the continuum setting, is of order $\ep^2|\log\ep|$, while short dipoles of opposite dislocations induce a much smaller energy of order $\e^2$. 
The aim of the present paper  is to determine the asymptotic behavior (as $\ep\to 0$) of $F_\ep$ in the energy regime $\ep^2|\log\ep|$\,. 
This corresponds to a finite distribution of {\it geometrically necessary edge dislocations}, i.e., to the superposition of a finite number of isolated dislocations plus 
clusters of singularities with total Burgers vector equal to zero (usually referred to as {\it statistically stored dislocations}). 

Our results are obtained within the rigorous formalism of $\Gamma$-convergence and consist in a compactness property for the dislocation measures and in the derivation of the effective limit energy induced by dislocations. 
Specifically, 
our analysis shows that the discrete dislocation density generated by the plastic slip $\sigma$ converges, in the sense of flat convergence \eqref{defflat}, to a finite sum of Dirac deltas $\sum_{k=1}^Kb^k\delta_{x^k}$
with $b^k\in\Tl$\,.  The effective energy (namely the $\Gamma$-limit) ``counts'' the limiting singularities with a coefficient $\ffi(b^k)$ given by the explicit formula \eqref{specffi}. These results  rely on a similar  analysis performed 
 in \cite{GLP,DGP} in a {\it semi-discrete} setting within the so-called core radius approach.

We restrict our analysis to configurations of plastic slips inducing dislocations with  minimal mutual distance larger than or equal to $\sqrt 3 \ep$\,. Roughly speaking, this means that  
two neighboring triangles cannot both contain a dislocation. 
Such a
mild separation assumption \eqref{MSass} guarantees that each dislocation induces a core energy of order $\e^2$;  the latter estimate is essential  also in the core radius approach, and represents the starting point in the so-called ball construction technique \cite{SS2}, 
which is devised to provide sharp lower bounds. 
This is not a mere technical assumption; indeed, removing it, one can exhibit  unphysical  configurations  with zero stored elastic energy where all triangles contain a dislocation (see Remark \ref{counter-MS}).  
In this respect, it seems that our linearized model (without assuming \eqref{MSass}) fails to describe the core energy stored in a single triangle and induced by  the presence of a dislocation. On the other hand, under the mild separation assumption \eqref{MSass}, each dislocation is surrounded by an annulus of 
 ``elastically deformed''
triangles where a finite amount of  energy is stored, as it follows from an application of Korn's inequality.    

Although our model provides a good description of ``non-pathological'' dislocation configurations, it exhibits some degeneracies,  due to the discrete linearized framework relying on a reference configuration, as well as to the presence of the slip variable (see Section \ref{finalcomment}).
A natural way to rule out such degeneracies could be to include kinematic constraints on the slip fields mimicking (in a  discrete framework) pure shear/deviatoric stress conditions that are typically assumed in (continuum) elasto-plasticity. In this respect, in Section \ref{finalcomment}, we discuss a possible constraint on the slip fields, based on the (formal) linearization of the volume preserving condition on dislocation-free $\ep$-triangles; we remark that our $\Gamma$-convergence analysis still holds true under such a constraint.

Finally, some comments are in order. First, we do believe that a similar discrete-to-continuum analysis could be developed in the $\ep^2|\log\ep|^2$ regime, which corresponds to a limit diffused distribution of dislocations and to an effective energy accounting for both elastic and plastic effects (see \cite{Gi} for an analogous result within the core-radius approach). Second, it would be interesting to look at the coarse-graining for the energy $F_\ep$ in even higher regimes, such as that of the grain boundaries (see the recent papers \cite{GT, LL, FPP}). Third, 
a challenging issue would be 
to deal directly with nonlinear models built on the deformed configurations. An intriguing intermediate attempt  in this direction could be to consider {\it hybrid models} 
combining the mathematical efficiency of linearized theories together with the mechanical understanding of dislocations in terms of interpenetrating pentagon-heptagon pairs usually observed in the deformed configurations.
\bigskip
\paragraph{\bf Notation.} For every $m,n\in\N$ and for every matrix $M\in\R^{m\times n}$\,, we denote by $M^*\in\R^{n\times m}$ the transpose matrix of $M$\,, i.e.,  such that $M^{*}_{ij}=M_{ji}$ for every $i=1,\ldots,n$ and $j=1,\ldots,m$\,.
In what follows the elements $(x_1,x_2)$ of $\R^2$ will be also  identified (whenever it is convenient) with column vectors  $\left(\begin{array}{l}x_1\\ x_2\end{array}\right)=(x_1\quad x_2)^*$. 
We denote by $\{e_1,e_2\}$ the canonical basis of $\R^2$\,, where $e_1=(1,0)$ and $e_2=(0, 1)$\,. 
Given two vectors $a=(a_1,a_2)\,,\ b=(b_1,b_2)\in\R^2$\,, we set $a\wedge b:=a_1b_2-a_2b_1$\,;
moreover,
we set $a^\perp:=(-a_2, a_1)$\,. 

For every open set $U$ with $\partial U$ smooth and for every $s\in\partial U$\,, we denote by $\tau(s)$ the tangent vector to $\partial U$ at $s$\, defined as $\tau(s) = n^\perp(s)$, where $n(s)$ denotes the outer normal unitary vector to $\partial U$ at $s$. 

For every $p\in\R^2$ and for every $0<r<R$ we define the annulus $A_{r,R}(p):=B_R(p)\setminus \overline{B}_r(p)$\,, where, for every $\rho>0$\,, $B_\rho(p)$ and $\overline{B}_\rho(p)$ denote the open and the closed ball centered at $p$ with radius $\rho$\,, respectively.

Moreover, for every bounded open set $A\subset\R^2$ and for every $\beta\in L^2(A;\R^{2\times 2})$ the symbol $\mathrm{Curl}\,\beta$ denotes the row-by-row distributional curl of $\beta$\,; formally,
$$
\mathrm{Curl}\,\beta=(\partial_{x_1}\beta_{12}-\partial_{x_2}\beta_{11},\partial_{x_1}\beta_{22}-\partial_{x_2}\beta_{21})\,.
$$
Analogously, the symbol $\mathrm{Div}\,\beta$ will denote the row-by-row distributional divergence of $\beta$\,, formally given by
$$
(\mathrm{Div}\,\beta)_i=\partial_{x_1}\beta_{i1}+\partial_{x_2}\beta_{i2}\,,\qquad i\in\{1,2\}\,.
$$ 
Finally, $\M(\R^2;\R^2)$ denotes the class of $\R^2$ valued Radon measures on $\R^2$\,.

\section{The model and the main result}\label{sc:model}
In this section we introduce our model and state the main result.\\
\paragraph{\bf Reference lattice.} 
We set
$\nu:=\frac{1}{2} e_1+\frac{\sqrt 3}{2} e_2$ and $\eta:=-\frac 1 2 e_1+\frac{\sqrt 3}{2} e_2$\,.
Let $\Tl:=\Span_\Z\{e_1,\nu\}$ and set 
$$
T^+:=\conv\{0, e_1,\nu\}\quad\textrm{ and }\quad T^-:=\conv\{0, e_1,-\eta\}\,,
$$
where, for every $a,b,c\in\R^2$, the set $\conv\{a,b,c\}$ denotes the convex envelope of the points $a$, $b$, $c$, i.e., the (closed) triangle with vertices at $a$, $b$, $c$\,.
For every $\ep>0$ we denote by $\T_\ep$ the family of the triangles  $T_\ep$ of the form $i+\ep T^{\pm}$, with $i\in\ep\Tl$\,.

Let $\Om\subset\R^2$ be a bounded open set with Lipschitz continuous boundary, representing the domain of definition of the relevant fields in the model.
For every $\ep>0$, we set
$$
\T_\ep(\Omega):=\{T_\ep\in\T_\ep\,:\,T_\ep\subset\Omega\}
$$ 
and 
we define
 $\Omega_{\ep}:=\bigcup_{T_\ep\in\T_\ep(\Omega)}T_\ep$. Moreover, we set $\Omega_\ep^0:=\Omega_\ep\cap \ep\Tl$ and we denote by $\Omega_{\ep}^1$ the family of nearest neighbor bonds in $\Omega_\ep$, i.e., $\Omega_{\ep}^1:=\{(i,j)\in\Omega_\ep^0\times \Omega_\ep^0\,:\,|i-j|=\ep\}$. Trivially, $(i,j)\in\Omega_\ep^1$ if and only if $(j,i)\in\Omega_\ep^1$\,.
 
In the following we will generalize the notation introduced above to general subsets of $\R^2$ (not necessarily open). In particular, for every triangle $T_\ep\in\T_\ep$\,, we have
\begin{equation*}
(T_\ep)_\ep^1=\{(i,j)\in (T_\ep\cap \ep\Tl)\times (T_\ep\cap\ep\Tl)\,:\,i\neq j\}\,.
\end{equation*}
For every map $V:(T_\ep)_\ep^1\to\R^2$\,, we define the {\it discrete circulation} of $V$ on the ``boundary of $T_\ep$'' as
\begin{equation}\label{circudisc}
\ud V(T_\ep):=V(i,j)+V(j,k)+V(k,i)\,,
\end{equation}
where $(i,j,k)$ is a triple of  counter-clockwise oriented vertices of $T_\ep$\,.
\bigskip
\paragraph{\bf Displacement, plastic slip, dislocation measure.}
We denote by $\AD_\ep(\Omega)$ the class of displacements $u:\Omega_\ep^0\to \R^2$; moreover, for every $u\in\AD_\ep(\Omega)$ we define the {\it discrete gradient} $\ud u:\Omega_\ep^1\to \R^2$  of $u$ as $\ud u(i,j):=u(j)-u(i)$ for every $(i,j)\in \Omega_\ep^1$. 

We define the class of {\it plastic slips} $\AS_\ep(\Omega)$ as
\begin{equation*}
\AS_\ep(\Omega):=\Big\{\sigma:\Omega_{\ep}^1\to\ep\Tl\,:\,\sigma(i,j)=-\sigma(j,i)\textrm{ for every }(i,j)\in\Omega_\ep^1\Big\}\,.
\end{equation*}
For every $\sigma\in\AS_\ep(\Omega)$ we define the discrete dislocation measure associated to $\sigma$ as
 $$
 \mu[\sigma]:=-\sum_{T_\ep\in \T_\ep(\Omega)}\ud\sigma(T_\ep)\delta_{x_{T_\ep}}\,, 
 $$ 
 where $\ud\sigma$ is defined in \eqref{circudisc} and $x_{T_\ep}$ denotes the barycenter of the triangle $T_\ep$\,.
 Notice that for every $u\in\AD_\ep(\Omega)$ and for every $\sigma\in\AS_\ep(\Omega)$ it holds
 $$
 \mu[\sigma]=\sum_{T_\ep\in \T_\ep(\Omega)}\ud(\ud u-\sigma)(T_\ep)\delta_{x_{T_\ep}}\,.
 $$
 
The class of {\it admissible dislocation measures}, denoted by  $X_{\ep}(\Omega)$, is the family of measures of the form $\mu=\sum_{T_\ep\in \T_\ep(\Omega)}b(T_\ep)\delta_{x_{T_\ep}}$ with $b(T_\ep)\in \ep\Tl$ and satisfying the following {\it mild separation} property:
\begin{align}\label{MSass}\tag{{MS}}
&\textrm{for every $T_\ep\in\T_\ep(\Omega)$ with $\mu(T_\ep)\neq 0$, we have }\\
\nonumber
&\textrm{$\partial T_\ep \cap \partial \Omega_\ep = \emptyset$ and $\mu(T'_\ep)=0$ for every $T'_\ep\in\T_\ep(\Omega)$ with $\partial T'_\ep \cap \partial T_\ep \neq \emptyset$.}
\end{align}
Finally, 
we set
 $$
 X(\Omega):=\Big\{\mu=\sum_{k=1}^{K}b^k\delta_{x^k}\,:\,K\in\N\,,\, b^k\in\Tl\,,\, x^k\in\Omega \Big\}\,.
 $$
 
\bigskip
\paragraph{\bf The energy functional and the main result.}
We are now in a position to define the energy functionals $F_\ep:\AD_\ep(\Omega)\times\AS_{\ep}(\Omega)\to [0,+\infty)$ as
\begin{equation}\label{def:en}
F_\ep(u,\sigma):=\frac {1}{2\ep^2}\sum_{(i,j)\in\Omega_\ep^1}[(\ud u(i,j)-\sigma(i,j))\cdot (j-i)]^2\,,
\end{equation}
where $\cdot$ denotes the standard scalar product in $\R^2$\,.

We will consider also localized versions of the functional $F_\ep(u,\sigma)$ in \eqref{def:en}. More specifically, for every set $A\subset\R^2$\,, we define $F_\ep(\cdot,\cdot;A):\AD_\ep(A)\times\AS_{\ep}(A)\to [0,+\infty)$ as
\begin{equation}\label{enloc}
F_\ep(u,\sigma;A):=\frac {1}{2\ep^2}\sum_{(i,j)\in A_\ep^1}[(\ud u(i,j)-\sigma(i,j))\cdot (j-i)]^2\,,
\end{equation}
so that $F_\ep(u,\sigma;\Omega)=F_\ep(u,\sigma)$\,.

Since in our analysis the relevant parameter is  the dislocation measure $\mu$ associated to $\sigma$, we let the energy functionals depend only on (the admissible measures) $\mu$, by defining $\F_\ep:X_\ep(\Omega)\to [0,+\infty)$ as 
\begin{equation*}
 \F_\ep(\mu):=\inf_{\newatop{(u,\sigma)\in\AD_\ep(\Omega)\times\AS_{\ep}(\Omega)}{\mu[\sigma]=\mu}}F_\ep(u,\sigma)\,.
 \end{equation*}
For every $b\in\Tl$, we set
\begin{equation}\label{ffi}
\ffi(b):= \frac{1}{3\pi}\min\bigg\{\sum_{i=1}^3|z_i|\,:\, z_1,z_2,z_3\in\Z\,,\quad b=z_1e_1+z_2\nu+z_3\eta\bigg\}\,.
\end{equation}
 
Our main result is the following theorem.
\begin{theorem}\label{mainthm}
The following $\Gamma$-convergence result holds true.
\begin{itemize} 
\item[(i)] (Compactness) Let $\{\mu_\ep\}_\ep\subset \M(\R^2;\R^2)$ be such that $\mu_\ep\in X_\ep(\Omega)$ for every $\ep>0$\,. If $\F_\ep(\mu_\ep)\le C\ep^2|\log\ep|$\,, then, up to a subsequence, $\frac{\mu_\ep}{\ep}\fla\mu$ (as $\ep\to 0$) for some $\mu\in X(\Omega)$\,.
\item[(ii)] ($\Gamma$-liminf inequality) For every $\mu=\sum_{k=1}^{K}b^k\delta_{x^k}\in X(\Omega)$ and for every $\{\mu_\ep\}_\ep\subset\M(\R^2;\R^2)$ with $\mu_\ep\in X_\ep(\Omega)$ for every $\ep>0$ and such that $\frac{\mu_\ep}{\ep}\fla\mu$ (as $\ep\to 0$) it holds
\begin{equation}\label{form:liminf}
 \liminf_{\ep\to 0}\frac{\F_\ep(\mu_\ep)}{\ep^2|\log\ep|}\ge \frac{\sqrt 3}{2}\sum_{k=1}^K\ffi(b^k) \,.
\end{equation} 
\item[(iii)] ($\Gamma$-limsup inequality) For every $\mu=\sum_{k=1}^{K}b^k\delta_{x^k}\in X(\Omega)$ there exists $\{\mu_\ep\}_\ep\subset\M(\R^2;\R^2)$ such that $\mu_\ep\in X_\ep(\Omega)$ for every $\ep>0$\,, $\frac{\mu_\ep}{\ep}\fla\mu$ (as $\ep\to 0$) and
\begin{equation}\label{form:limsup}
\limsup_{\ep\to 0}\frac{\F_\ep(\mu_\ep)}{\ep^2|\log\ep|}\le \frac{\sqrt 3}{2}\sum_{k=1}^K\ffi(b^k)\,.
\end{equation}
\end{itemize} 
\end{theorem} 
The convergence appearing in Theorem \ref{mainthm} is the {\it flat convergence}\,, that is the convergence with respect to the {\it flat norm}, defined by
\begin{equation}\label{defflat}
\|\mu\|_{\flt}:=\sup_{\newatop{\phi\in C^{0,1}_{\mathrm{c}}(\Omega)}{\|\phi\|_{C^{0,1}}\le 1}}\Big|\int_{\Omega}\phi\,\ud\mu\Big|\,,\qquad\textrm{for every }\mu\in\M(\R^2;\R^2)\,,
\end{equation}
where $C^{0,1}(\Omega)$ is the space of Lipschitz continuous functions endowed with the norm
$$
\|\phi\|_{C^{0,1}}:=\sup_{x\in\Omega}|\phi(x)|+\sup_{\newatop{x,y\in\Omega}{x\neq y}}\frac{|\phi(x)-\phi(y)|}{|x-y|}\,,
$$
 and $C^{0,1}_{\mathrm{c}}(\Omega)$ is the subspace of $C^{0,1}$ functions compactly supported in $\Omega$\,.
\begin{remark}\label{linearizzazione}
We illustrate how to formally derive \eqref{def:en} from a nonlinear frame invariant model with nearest neighbor pairwise interaction potentials, in the case where no dislocation is present, i.e., $\sigma\equiv0$.  
We assume that, given a deformation  $v\colon \Om_\e^0 \to \R^2$, the total interaction energy is
\[
E_\e(v)=\e^2 \sum_{(i,j)\in\Omega_\ep^1}\psi \Big( \Big|\frac{\ud v(i,j)}{\e}\Big|\Big)\,,
\]
\noindent
where $\psi$ is a $C^2$ function such that $\psi(1)=\psi'(1)=0$ and $\psi''(1)>0$\,, so that the identity is an equilibrium configuration. Expressing the energy in terms of the displacement 
$u(x)=\frac{v(x)- x}{\delta}$, with $\delta>0$, and rescaling by $\delta^2$, the total 
interaction energy reads as  
\[
E^\delta_\e(x+\delta u)=\frac{\e^2}{\delta^2} \sum_{(i,j)\in\Omega_\ep^1}\psi \Big( \Big|\frac{j-i}{\e}+\delta\, \frac{\ud u(i,j)}{\e}\Big|\Big)\,.
\]
A second order Taylor expansion of $E^\delta_\e(x+\delta u)$ with respect to $\delta$ about
the point $\delta=0$ gives
$$
E_\e^\delta(x+\delta u)= F_\e(u)+\mathrm{o}_{\delta}(1)\,,
$$
where
\begin{equation*}
F_\e(u)=\frac{1}{2\e^2}\psi''(1)\sum_{(i,j)\in\Omega_\ep^1}[\ud u(i,j)\cdot (j-i)]^2\,,
\end{equation*}
which coincides with \eqref{def:en} up to a multiplicative constant when $\sigma\equiv0$.
We refer to \cite{GT} for a formal derivation of \eqref{def:en} starting from a nonlinear model with discrete plastic slips. 
\end{remark}
\section{The continuum case}
In order to prove Theorem \ref{mainthm} we will use the corresponding result in the continuum setting \cite{DGP}, that we briefly recall here.
 Let $A$ be an open bounded subset of $\R^2$ with Lipschitz continuous boundary. For every $\mu\in X(A)$ we set \begin{equation}\label{dombuc}
 A_\ep(\mu):=A\setminus\bigcup_{x\in \supp(\mu)}\overline{B}_\ep(x)
 \end{equation}
  and we define the set $\as_\ep(\mu)$ of admissible strains associated to $\mu$ as
\begin{equation}\label{def:admstr}
\begin{aligned}
\as_\ep(\mu):=&\Big\{\beta\in L^2(A_\ep(\mu);\R^{2\times 2})\,:\,\mathrm{Curl}\,\beta=0\,\textrm{ in }A_\ep(\mu)\,,\\
&\int_{\partial U}\beta\,\tau\,\ud\mathcal{H}^1=\mu(U)\textrm{ for every open set }U\subset A
\textrm{ with }\partial U\subset A_\ep(\mu)\textrm{ smooth}\Big\}\,.
\end{aligned}
\end{equation}

Let $\C$ be an elasticity tensor, i.e., a linear operator from $\R^{2\times 2}$ into itself satisfying the following property: There exist two constants $0<c_1<c_2$ such that
\begin{equation*}
c_1|\beta^{\sym}|^2\le \frac 1 2 \C\beta:\beta\le c_2|\beta^{\sym}|^2\qquad\textrm{for every }\beta\in\R^{2\times 2}\,,
\end{equation*}
where $\beta^{\sym}:=\frac 1 2(\beta+\beta^*)$\,.

The elastic energy of a field $\beta\in\as_\ep(\mu)$ in the body $A$ is given by
\begin{equation*}
E_\ep(\beta;A_\ep(\mu)):=\frac{1}{2}\int_{A_\ep(\mu)}\C\beta:\beta\,\ud x\,,
\end{equation*}
and the energy induced by the dislocation distribution $\mu$ in the body $A$
is defined by 
\begin{equation}\label{enemin}
\E_\ep(\mu;A):=\inf_{\beta\in\as_\ep(\mu)}E_\ep(\beta;A_\ep(\mu))+|\mu|(A)\,,\qquad\textrm{for every }\mu\in X(A) \,.
\end{equation}
In \eqref{enemin}\,, the first addendum on the right hand side is the elastic energy induced by the dislocation measure $\mu$\,, whereas the second addendum plays the role of a plastic core energy.

In order to introduce the self-energy of an edge dislocation, for every $b\in\Tl$ we first define the strain field $\beta_{\R^2}^{b,\C}$ satisfying the circulation condition
\begin{equation*}
\mathrm{Curl}\,\beta=b\delta_{0}\textrm{ in }\R^2
\end{equation*}
and the equilibrium equation
\begin{equation*}
\mathrm{Div}\,\C\beta=0\textrm{ in }\R^2\,.
\end{equation*} 
As shown in \cite{DGP}, the strain field $\beta_{\R^2}^{b,\C}$ is given, in polar coordinates, by
\begin{equation}\label{betapiano}
\beta^{b,\C}_{\R^2}(\rho,\theta):=\frac{1}{\rho}\big(f^{b,\C}(\theta)\otimes(-\sin\theta,\cos\theta)+g^{b,\C}\otimes(\cos\theta,\sin\theta)\big)\,, 
\end{equation}
where  the constant $g^{b,\C}\in\R^2$ and the function $f^{b,\C}\in C^0([0,2\pi];\R^2)$\,, with $f(0)=f(2\pi)$ and $\int_{0}^{2\pi}f^{b,\C}(\omega)\,\ud\omega=b$\,, are uniquely determined by the vector $b$ and the tensor $\C$\,.

The corresponding displacement $u^{b,\C}_{\R^2}$ (i.e., such that $\nabla u^{b,\C}_{\R^2}=\beta^{b,\C}_{\R^2}$)
is computed explicitly in the literature (see, for instance, \cite[formula 4.1.25]{BBS}) and, in polar coordinates,
is given by
\begin{equation}\label{uottima}
u^{b,\C}_{\R^2}(\rho,\theta)=F^{b,\C}(\theta)+g^{b,\C}\log\rho\,,
\end{equation}
where  $F^{b,\C}(\theta)=\int_{0}^{\theta}f^{b,\C}(\omega)\,\ud\omega$\,, for $\theta\in[0,2\pi)$.
Note that the displacement above is uniquely determined up to a constant.

For every $b\in\Tl$ we set 
\begin{equation}\label{servealimsup}
\psi^{\C}(b):=\int_{0}^{2\pi}\frac{1}{2}\C\,\Gamma^{b,\C}(\theta):\Gamma^{b,\C}(\theta)\,\ud\theta
=\frac{1}{|\log r|}\int_{B_1\setminus\overline{B}_r}\frac{1}{2}\C\beta^{b,\C}_{\R^2}:\beta^{b,\C}_{\R^2}\,\ud x\,,\qquad  0<r<1\,,
\end{equation}
where we have set $\Gamma^{b,\C}(\theta):=\big(f^{b,\C}(\theta)\otimes(-\sin\theta,\cos\theta)+g^{b,\C}\otimes(\cos\theta,\sin\theta)\big)$ for every $\theta\in [0,2\pi]$\,.
Finally, for every $b\in\Tl$ we define 
\begin{equation}\label{generffi}
\ffi^{\C}(b):=\min\bigg\{\sum_{i=1}^{N}|z_i|\psi^{\C}(b_i)\,:\, z_i\in\Z\,,\, b_i\in\Tl\,,\,N\in\N\,,\,\sum_{i=1}^{N}z_ib_i=b\bigg\}\,.
\end{equation}
The following result is a slight variant of \cite[Theorem 4]{DGP}\,.
\begin{theorem}\label{thmcont}
The following $\Gamma$-convergence result holds true.
\begin{itemize}
\item[(i)] (Compactness) 
 Let $\{\mu_\ep\}_{\ep}\subset X(A)$ and let $\{\beta_\ep\}_{\ep}$ be a sequence of fields with $\beta_\ep\in\as_\ep(\mu_\ep)$ (for every $\ep>0$) such that
\begin{equation}\label{unifbound}
E_\ep(\beta_\ep;A_\ep(\mu_\ep))+|\mu_\ep|(A)\le M|\log\ep|\,;
\end{equation}
then, up to a subsequence, $\mu_\ep\fla\mu$ (as $\ep\to 0$) for some $\mu\in X(A)$\,.
\item[(ii)] ($\Gamma$-liminf inequality) For every $\mu=\sum_{k=1}^Kb^k\delta_{x^k}\in X(A)$\,, for every $\{\mu_\ep\}_{\ep}\subset X(A)$ with $\mu_\ep\fla\mu$ (as $\ep\to 0$) and $|\mu_\ep|(A)\le C|\log\ep|$ and for every sequence of fields $\{\beta_\ep\}_{\ep}$ with $\beta_\ep\in\as_\ep(\mu_\ep)$ (for every $\ep>0$)  it holds
\begin{equation}\label{lbcont}
\liminf_{\ep\to 0}\frac{E_\ep(\beta_\ep;A_\ep(\mu_\ep))}{|\log\ep|}\ge\sum_{k=1}^K\ffi^{\C}(b^k)\,;
\end{equation}
\item[(iii)] ($\Gamma$-limsup inequality) For every $\mu=\sum_{k=1}^Kb^k\delta_{x^k}\in X(A)$ there exists a sequence of measures $\{\mu_\ep\}_{\ep}\subset X(A)$ such that $\mu_\ep\fla\mu$ (as $\ep\to 0$) and
\begin{equation}
\limsup_{\ep\to 0}\frac{\E_\ep(\mu_\ep;A)}{|\log\ep|}\le \sum_{k=1}^K\ffi^{\C}(b^k)\,. 
\end{equation}
\end{itemize}
\end{theorem}
The proof of Theorem \ref{thmcont} can be obtained by arguing verbatim as in the proof of \cite[Theorem 4]{DGP}. In fact, the minor differences between Theorem \ref{thmcont} and \cite[Theorem 4]{DGP} are the following.
First, in \cite{DGP} the admissible strains $\beta$ should satisfy the condition 
\begin{equation*}
\int_{A_\ep(\mu)}(\beta-\beta^*)\,\ud x=0\,.
\end{equation*}
Such a condition can be always enforced in view of the invariance of the elastic energy with respect to translations. 
Second, in \cite[Theorem 4]{DGP} the compactness property is stated enforcing \eqref{unifbound} only for the optimal $\beta_\ep$.
Third, in \cite[Theorem 4]{DGP}, the lower bound \eqref{lbcont} is provided for the functional $\E_\ep$; there, the assumption $|\mu_\ep|(A)\le C|\log\ep|$ is automatically satisfied by sequences with equibounded energy. However, assuming $|\mu_\ep|(A)\le C|\log\ep|$\,, the same proof provides the same lower bound for $\frac{1}{|\log\ep|}E_\ep(\beta_\ep;A_\ep(\mu_\ep))$\,.
 
Now, we specialize the functions $\psi^{\C}$ and $\ffi^{\C}$ to the particular isotropic case we deal with in the discrete-to-continuum limit.
More specifically, let $\C$ be the isotropic elasticity tensor with Lam\'e parameters both equal to $1$\,, 
i.e.,
\begin{equation}\label{isotensor}
\C \beta:\beta:=|\mathrm{tr}\,\beta|^2+2 |\beta^{\sym}|^2\,,
\end{equation}
By \cite[formulas (3.3) and (3.4)]{CL}, for the specific choice of $\C$ in \eqref{isotensor} we have that the constant $g^{b,\C}$ and the function $f^{b,\C}$ in \eqref{betapiano} are given by
\begin{equation}
\begin{aligned}\label{gfottime}
g^{b,\C}=\ &-\frac{1}{6\pi} b^{\perp}\,,\\
f^{b,\C}(\theta)=\ &\frac{1}{2\pi}b-\frac{1}{3\pi}(-b_1\cos(2\theta)-b_2\sin(2\theta), b_2\cos(2\theta)-b_1\sin(2\theta))\,;
\end{aligned}
\end{equation}
therefore, by straightforward computations, in this case
\begin{equation*}
\psi(b):=\psi^{\C}(b)=\frac{1}{3\pi}|b|^2\,,\qquad b\in\Tl\,,
\end{equation*}
and hence, by \eqref{generffi}, 
\begin{equation}\label{specffi}
\begin{aligned}
\ffi(b)=\ffi^{\C}(b)=\ &\frac{1}{3\pi}\min\bigg\{\sum_{i=1}^N|z_i|\,|b_i|^2\,:\, z_i\in\Z\,,\, b_i\in\Tl\,,\,N\in\N\,,\,\sum_{i=1}^{N}z_ib_i=b\bigg\}\\
=\ &\frac{1}{3\pi}\min\bigg\{\sum_{i=1}^3|z_i|\,:\, z_1,z_2,z_3\in\Z\,,\quad b=z_1e_1+z_2\nu+z_3\eta\bigg\}\,,
\end{aligned}
\end{equation}
which is exactly \eqref{ffi}.
\section{Proof of the main result}
This section is devoted to the proof of Theorem \ref{mainthm}. In Subsection \ref{preli} below we collect some auxiliary results that will be instrumental to prove Theorem \ref{mainthm}.
\subsection{Preliminary results}\label{preli}
We start by deriving the continuum isotropic elasticity tensor associated to our discrete functional.
\begin{remark}\label{disccont}
Let $\ep>0$ and let $T_\ep\in\T_\ep$\,. Let moreover $u:T_\ep\cap\ep\Tl\to \R^2$ and $\sigma: (T_\ep)_\ep^1\to\ep\Tl$ be such that $\sigma(i,j)=-\sigma(j,i)$ for every $(i,j)\in (T_\ep)_\ep^1$\,.
If $\beta\in\R^{2\times 2}$ satisfies 
\begin{equation*}
\beta\,(j-i)=\ud u(i,j)-\sigma(i,j)\qquad\textrm{for every }(i,j)\in (T_\ep)_\ep^1,
\end{equation*}
then, by straightforward computations, we have that 
\begin{equation}\label{ok1}
\begin{aligned}
F_\ep(u,\sigma;T_\ep)
=&\ep^2\Big(|e_1^*\,\beta\,e_1|^2+|\nu^*\,\beta\,\nu|^2+|\eta^*\,\beta\,\eta|^2\Big)
= \ep^2 \frac{3}{8}\C\beta:\beta
=\frac{\sqrt 3}{2}\int_{T_\ep}\C\beta:\beta\,\ud x\,,
\end{aligned}
\end{equation}
where $F_\ep(\cdot,\cdot;T_\ep)$ is defined in \eqref{enloc} and $\C$ is given in \eqref{isotensor}.
\end{remark}
In the next lemma we construct, far from the singularities, a (continuous) strain field $\beta$ that is compatible with a given distribution of dislocations and whose (continuous) energy behaves like the discrete energy $F_\ep$\,. Such an estimate, together with a bound on the total variation, 
will allow us to deduce the $\Gamma$-liminf inequality in Theorem \ref{mainthm}(ii) directly from the analogous statement in the continuous setting (Theorem \ref{thmcont}(ii)).
\begin{lemma}\label{poincare}
Let $\ep>0$ and let $(u,\sigma)\in\AD_\ep(\Omega)\times\AS_\ep(\Omega)$\,.
Let moreover
\begin{equation}\label{misnul}
\mathcal{K}_\ep:=\{T_\ep\in\T_\ep(\Omega)\,:\,\mu[\sigma](T_\ep)=0\}\,,\qquad K_\ep:=\bigcup_{T_\ep\in\mathcal{K}_\ep}T_\ep\,.
\end{equation}
Then, there exists a piecewise constant (namely, constant on each triangle $T_\ep\in\mathcal{K}_\ep$) field $\beta^{u,\sigma,\mathcal{K}_\ep}\in L^2(K_\ep;\R^{2\times 2})$ such that:
\begin{itemize}
\item[(a)] $\ud u(i,j)-\sigma(i,j)=\beta^{u,\sigma,\mathcal{K}_\ep}_{|_{T_\ep}}\,(j-i)$ for every $i,j\in T_\ep$ with $T_\ep\in\mathcal{K}_\ep$\,;
\item[(b)] $F_\ep(u,\sigma; K_\ep)=\frac{\sqrt 3}{2}\int_{K_\ep} \frac 1 2\C\beta^{u,\sigma,\mathcal{K}_\ep}:\beta^{u,\sigma,\mathcal{K}_\ep}\,\ud x+\frac{1}{4\ep^2}\sum_{\newatop{(i,j)\in K^1_\ep}{i,j\in\partial K_\ep}}|(\ud u(i,j)-\sigma(i,j))\cdot (j-i)|^2$\,, 
where $F_\ep$ is defined in \eqref{enloc} and $\C$ is defined in \eqref{isotensor}\,;
\item[(c)] $\mathrm{Curl}\,\beta^{u,\sigma,\mathcal{K}_\ep}=0$ in $K_\ep$ and $\int_{\partial U}\beta^{u,\sigma,\mathcal{K}_\ep}\,\tau\,\ud\mathcal{H}^1=\mu[\sigma](U)$ for every smooth open set  $U\subset\Omega$  such that $\partial U\subset K_\ep$\,.
\end{itemize}
\end{lemma}
\begin{proof}
Let $T_\ep=\mathrm{conv}\{i,j,k\}\in\mathcal{K}_\ep$ and assume that the triple $(i,j,k)$
is counterclockwise oriented. Let $v_{T_\ep}\in\AD_{\ep}(T_\ep)$ be defined by 
$$
v_{T_\ep}(i)=0\,,\quad v_{T_\ep}(j)=\ud u(i,j)-\sigma(i,j)\,,\quad v_{T_\ep}(k)=v_{T_\ep}(j)+\ud u(j,k)-\sigma(j,k)\,,
$$
and notice that the discrete gradient $\ud v_{T_\ep}$ of $v_{T_\ep}$ agrees with $\ud u-\sigma$\,.
For every $T_\ep\in\mathcal{K}_\ep$ let $\tilde v_{T_\ep}:T_\ep\to\R^2$ be the piecewise affine interpolation of $v_{T_\ep}$ and let $\beta^{u,\sigma,\mathcal{K}_\ep}: K_\ep\to \R^{2\times 2}$ be the map defined by
$$
\beta^{u,\sigma,\mathcal{K}_\ep}:=\sum_{T_\ep\in\mathcal{K}_\ep}\D \tilde v_{T_\ep}\,\chi_{T_\ep}\,,
$$
where $\D \tilde{v}_{T_\ep}$ denotes the (continuous) gradient of the function $\tilde v_{T_\ep}$\,. 
Clearly  $\beta^{u,\sigma,\mathcal{K}_\ep}\in L^2(K_\ep;\R^{2\times 2})$\,.
Moreover, by construction, we have that property (a) is satisfied; furthermore, since  for every $T_\ep\in\mathcal{K}_\ep$ and for every $(i,j)\in T_\ep^1$ it holds
\begin{equation*}
\beta^{u,\sigma,\mathcal{K}_\ep}_{|_{T_\ep}}\,(j-i)=\ud u(i,j)-\sigma(i,j)\,,
\end{equation*}
by Remark \ref{disccont} we get
\begin{equation*}
\begin{aligned}
F_\ep(u,\sigma;K_\ep)=\ &\frac 1 2\sum_{T_\ep\in\mathcal{K}_\ep}F_\ep(u,\sigma;T_\ep)+\frac{1}{4\ep^2}\sum_{\newatop{(i,j)\in K^1_\ep}{i,j\in\partial K_\ep}}|(\ud u(i,j)-\sigma(i,j))\cdot (j-i)|^2\\
=\ &\frac{\sqrt 3}{2}\int_{K_\ep} \frac 1 2\C\beta^{u,\sigma,\mathcal{K}_\ep}:\beta^{u,\sigma,\mathcal{K}_\ep}\,\ud x+\frac{1}{4\ep^2}\sum_{\newatop{(i,j)\in K^1_\ep}{i,j\in\partial K_\ep}}|(\ud u(i,j)-\sigma(i,j))\cdot (j-i)|^2\,,
\end{aligned}
\end{equation*}
i.e., property (b).

Finally, we prove that also (c) is satisfied. To this end, we notice that if $T_\ep^1$ and $T_\ep^2$ are two triangles sharing one edge $(i,j)$\,, then, by construction, $\beta^{u,\sigma,\mathcal{K}_\ep}_{|_{T_\ep^1}}\,(j-i)=\beta^{u,\sigma,\mathcal{K}_\ep}_{|_{T_\ep^2}}\,(j-i)$ so that $\mathrm{Curl}\,\beta^{u,\sigma,\mathcal{K}_\ep}=0$ on $K_\ep$\,.
Now, since for every $T_\ep\in \mathcal{K}_\ep$\,,
$$
\int_{\partial T_\ep}\beta^{u,\sigma,\mathcal{K}_\ep}\,\tau\,\ud \mathcal{H}^1= 0 \,,
$$
by using Stokes' Theorem, for every smooth open set  $U\subset\Omega$  such that $\partial U\subset K_\ep$, we have
$$
\int_{\partial U}\beta^{u,\sigma,\mathcal{K}_\ep}\,\tau\,\ud\mathcal{H}^1
= \int_{U\cap\partial K_\e}\beta^{u,\sigma,\mathcal{K}_\ep}\,\tau\,\ud\mathcal{H}^1
=-\sum_{T_\ep\subset U}\ud\sigma(T_\ep)=\mu[\sigma](U)\,,
$$
which concludes the proof of property (c) and of the whole lemma.
\end{proof}
The following result allows to extend the field $\beta$ constructed in Lemma \ref{poincare} above up to the boundary of $\Omega$\,. 
\begin{lemma}\label{lm:ext}
Let $\ep>0$ and let $(u,\sigma)\in\AD_\ep(\Omega)\times\AS_\ep(\Omega)$ be such that $\mu[\sigma]\in X_\ep(\Omega)$\,; let moreover $\mathcal{K}_\ep$ and $K_\ep$ be defined as in \eqref{misnul}
and let $\beta=\beta^{u,\sigma,\mathcal{K}_\ep}$ be the field provided by Lemma \ref{poincare}.
Then, there exists a field $\hat\beta=\hat\beta^{u,\sigma,\mathcal{K}_\ep}\in L^2(K_\ep\cup(\Omega\setminus\Omega_\ep);\R^{2\times 2})$
 such that
\begin{itemize}
\item[(i)] $\hat\beta=\beta$ in $K_\ep$\,;
\item[(ii)] $\mathrm{Curl}\,\hat\beta=0$ (in the sense of distributions);
\item[(iii)] $\int_{K_\ep\cup(\Omega\setminus\Omega_\ep)}\C\hat\beta:\hat\beta\,\ud x\le C\int_{K_\ep}\C\beta:\beta\,\ud x$\,, for some constant $C$ independent of $\ep$\,.
\end{itemize}
\end{lemma}
\begin{proof}
It is enough to notice that, in view of \eqref{MSass}, each of the triangles $T_\ep\in\T_\ep(\Omega)$ touching $\partial\Omega_\ep$ satisfies $\mu[\sigma](T_\ep)=0$\,. Therefore, in order to construct a field $\hat\beta$ satisfying properties (i), (ii), and (iii), it is enough to extend locally by reflection the field $\beta$ provided by Lemma \ref{poincare}. 
\end{proof}
We conclude this subsection showing how assumption \eqref{MSass} allows to estimate the total variation of the dislocation measure $\mu$ with the elastic energy $\F_\ep(\mu)$\,.
\begin{lemma}\label{lm:totvar}
There exists $\bar C>0$ such that  $|\frac{\mu}{\ep}|(\Omega)\le \bar C\frac{1}{\ep^2} \F_\ep(\mu)$ for every $\ep>0$ and for every $\mu\in X_\ep(\Omega)$\,. 
\end{lemma}
\begin{proof}
Let $\ep>0$ and $\mu\in X_\ep(\Omega)$ be fixed. Let moreover $(u,\sigma)\in\AD_\ep(\Omega)\times\AS_\ep(\Omega)$ with $\mu[\sigma]=\mu$\,. 
For every $p\in\supp\mu$ let $\mathcal{K}_{\ep}^p$ be the set of the triangles $T_\ep\neq T_\ep^p$ sharing at least one vertex with the triangle $T_{\ep}^p$ centered at $p$ and let $K_\ep^{p}:=\bigcup_{T_\ep\in\mathcal{K}_\ep^p}T_\ep$\,. By the very definition of $X_\ep(\Omega)$, we have that $\mu(T_\ep)=0$ for every $T_\ep\in\bigcup_{p\in\supp\mu}\mathcal{K}_\ep^p$.
 Let $\beta^{u,\sigma,\mathcal{K}_\ep^p}$ be the map defined in Lemma \ref{poincare} above. 
Recalling the notation for the annulus $A_{r,R}(p)$\,, for every $p\in\supp\mu$ we set
$$
S_{\ep}^p:=\frac{1}{|A_{\frac{\sqrt{3}}{3}\ep,\frac{5\sqrt{3}}{12}\ep}(p)|}\int_{A_{\frac{\sqrt{3}}{3}\ep,\frac{5\sqrt{3}}{12}\ep}(p)}(\beta^{u,\sigma,\mathcal{K}^p_\ep})^{\skw}\ud y\,,
$$
where, for every $\beta\in\R^{2\times 2}$\,, we have set $\beta^\skw:=\frac{1}{2}(\beta-\beta^*)$\,.
By Korn's inequality,  there exists a constant $\kappa$ (independent of $\ep$ and $p$) such that
\begin{equation*}
\begin{aligned}
 \int_{A_{\frac{\sqrt{3}}{3}\ep,\frac{5\sqrt{3}}{12}\ep}(p)}|(\beta^{u,\sigma,\mathcal{K}^p_\ep})^{\sym}|^2\ud x\ge\ & \kappa \int_{A_{\frac{\sqrt{3}}{3}\ep,\frac{5\sqrt{3}}{12}\ep}(p)}\big|\beta^{u,\sigma,\mathcal{K}^p_\ep}-S_\ep^p\big|^2\ud x \\
\ge\ &\kappa\int_{\frac{\sqrt 3}{3}\ep}^{\frac{5\sqrt{3}}{12}\ep}\int_{\partial B_\rho(p)}\big|(\beta^{u,\sigma,\mathcal{K}^p_\ep}-S_\ep^p)\,\tau\big|^2\,\ud\Huno\,\ud\rho\\
\ge\ &
\frac{\kappa}{2\pi}\log 2|\mu(p)|^2
= \frac{\kappa}{2\pi}\log 2\,\ep^2\Big|\frac{\mu}{\ep}(p)\Big|^2
\ge  \frac{\kappa}{2\pi}\log 2\,\ep^2\Big|\frac{\mu}{\ep}(p)\Big|\,,
\end{aligned}
\end{equation*}
where 
the second inequality follows by Fubini theorem, estimating $|\beta^{u,\sigma,\mathcal{K}^p_\ep}-S_\ep^p|$ from below by its tangential part, and the third one is a consequence of Jensen's inequality using Lemma \ref{poincare}(c) together with
$$
\int_{\partial B_\rho(p)} S^p_\ep\,\tau\,\ud\Huno=0\,.
$$
By assumption \eqref{MSass}, Lemma \ref{poincare}(b) and the very definition of $\C$ in \eqref{isotensor}, using that the annuli $A_{\frac{\sqrt{3}}{3}\ep,\frac{5\sqrt{3}}{12}\ep}(p)$ are pairwise disjoint and contained in $K_\ep^p$\,, we thus deduce that
\begin{equation*}
\begin{aligned}
F_\ep(u,\sigma)\ge&\frac{\sqrt 3}{2} \sum_{p\in\supp\mu}\frac 1 2\int_{K_\ep^p} \C\beta^{u,\sigma,\mathcal{K}_\ep^p}:\beta^{u,\sigma,\mathcal{K}_\ep^p}\,\ud x\ge C\sum_{p\in\supp\mu} \int_{A_{\frac{\sqrt{3}}{3}\ep,\frac{5\sqrt{3}}{12}\ep}(p)}|(\beta^{u,\sigma,\mathcal{K}^p_\ep})^{\sym}|^2\ud x\\
\ge&   C\ep^2 \Big|\frac{\mu}{\ep}\Big|(\Omega)\,,
\end{aligned}
\end{equation*}
which, taking the infimum over the pairs $(u,\sigma)\in\AD_\ep(\Omega)\times\AS_\ep(\Omega)$ with $\mu[\sigma]=\mu$\,, provides the claim for $\bar C=\frac{1}{C}$\,.
\end{proof}
\begin{remark}\label{counter-MS}
Without the assumption \eqref{MSass}, Lemma \ref{lm:totvar} does not hold. 
As an example,  for every $\ep>0$\,, let
$u_\ep\equiv0$ and
  \begin{equation*}
  \sigma_\ep(i,j):=\left\{\begin{array}{ll}
  +\sqrt3\ep \,e_2&\textrm{if }j=i+\e e_1\,,\\
  -\sqrt3\ep \,e_2&\textrm{if }j=i-\e e_1\,,\\
  0&\textrm{elsewhere in }\Omega_\ep^1\,.
  \end{array}
  \right.
  \end{equation*}
Trivially, $|\frac{\mu[\sigma_\ep]}{\ep}|(\Omega)\sim \frac{1}{\ep^2}$, $\big\|\frac{\mu[\sigma_\ep]}{\ep}\big\|_{\flt}\sim\frac{1}{\ep}$, and
$\F_\ep(\mu[\sigma_\ep])=F_\ep(u_\ep,\sigma_\ep)\equiv 0$\,. \end{remark}
\subsection{Proof of Theorem \ref{mainthm}}
We are now in a position to prove Theorem \ref{mainthm}.
\begin{proof}[Proof of Theorem \ref{mainthm}(i)]
For every $\ep>0$\,, let $(u_\ep,\sigma_\ep)\in\AD_\ep(\Omega)\times\AS_\ep(\Omega)$ with $\mu[\sigma_\ep]=\mu_\ep$ be such that
\begin{equation}\label{orainf2}
F_\ep(u_\ep,\sigma_\ep)\le 2\F_\ep(\mu_\ep)\,.
\end{equation}
We let
\begin{equation}\label{misnul2}
\mathcal{K}_\ep:=\{T_\ep\in\T_\ep(\Omega)\,:\,\mu_\ep(T_\ep)=0\}\,\qquad K_\ep:=\bigcup_{T_\ep\in\mathcal{K}_\ep}T_\ep\,,
\end{equation}
and we set  $\beta_\ep:=\beta^{u_\ep,\sigma_\ep,\mathcal{K}_\ep}$ and $\hat\beta_\ep:=\hat\beta^{u_\ep,\sigma_\ep,\mathcal{K}_\ep}$, where $\beta^{u_\ep,\sigma_\ep,\mathcal{K}_\ep}$ and  $\hat\beta^{u_\ep,\sigma_\ep,\mathcal{K}_\ep}$ are the fields provided by Lemmas \ref{poincare} and \ref{lm:ext}, respectively. Furthermore, we set $\tilde\mu_\ep:=\frac{\mu_\ep}{\ep}$ and $\tilde\beta_\ep:=\frac{\hat\beta_\ep}{\ep}$\,.
Recalling the definition of $\Omega_\ep(\tilde\mu_\ep)$ in \eqref{dombuc} and
 using, in order of appearance, Lemma \ref{lm:ext}(iii), Lemma \ref{poincare}(b), \eqref{orainf2} and the energy bound,
we have
\begin{equation}\label{unocomp}
\begin{aligned}
\frac 1 2 \int_{\Omega_\ep(\tilde\mu_\ep)}\C\tilde\beta_\ep:\tilde\beta_\ep\,\ud x=\ & \frac{1}{2\ep^2} \int_{\Omega_\ep(\tilde\mu_\ep)}\C\hat\beta_\ep:\hat\beta_\ep\,\ud x\le \frac{C}{2\ep^2}\int_{K_\ep} \C\beta_\ep:\beta_\ep\,\ud x\\
\le\ & \frac{C}{\ep^2} F_\ep(u_\ep,\sigma_\ep)\le\frac{C}{\ep^2}\F_\ep(\mu_\ep)\le C|\log\ep|\,;
\end{aligned}
\end{equation}
moreover, by Lemma \ref{lm:totvar} and by the energy bound, we have
\begin{equation}\label{duecomp}
|\tilde\mu_\ep|(\Omega)\le C|\log\ep|\,.
\end{equation}
In view of Lemma \ref{lm:ext}(ii) and of Lemma \ref{poincare}(c), we have $\tilde\beta_\ep\in\as_\ep(\tilde\mu_\ep)$ for every $\ep>0$ (with $A=\Omega$ in \eqref{def:admstr})\,, so that, by \eqref{unocomp} and \eqref{duecomp}, applying Theorem \ref{thmcont}(i) with $\mu_\ep=\tilde\mu_\ep$ and $\beta_\ep=\tilde\beta_\ep$, we deduce the claim.
\end{proof}
\begin{proof}[Proof of Theorem \ref{mainthm}(ii)]
We can assume without loss of generality that $\F_\ep(\mu_\ep)\le C\ep^2|\log\ep|$\,, which in view of Lemma \ref{lm:totvar}, implies that $
\big|\frac{\mu_\ep}{\ep}\big|(\Omega)\le C|\log\ep|$\,.

Let $A\subset\subset\Omega$ be an open bounded subset of $\R^2$ with Lipschitz continuous boundary such that $\supp\mu\subset A$\,. Let $\ep>0$ small enough such that $A\subset\Omega_{\ep}$\,.
Let moreover $(u_\ep,\sigma_\ep)\in\AD_\ep(\Omega)\times\AS_\ep(\Omega)$ with $\mu[\sigma_\ep]=\mu_\ep$ be such that
\begin{equation}\label{orainf3}
\F_\ep(\mu_\ep)\le F_\ep(u_\ep,\sigma_\ep)+\ep^2\,,
\end{equation} 
and let $\mathcal{K}_\ep$ be defined by \eqref{misnul2};
we set
 $\beta_\ep:=\beta^{u_\ep,\sigma_\ep,\mathcal{K}_\ep}$\,, where $\beta^{u_\ep,\sigma_\ep,\mathcal{K}_\ep}$ is the field provided by Lemma \ref{poincare}\,, $\tilde\mu_\ep:=\frac{\mu_\ep}{\ep}$ and $\tilde\beta_\ep:=\frac{\beta_\ep}{\ep}$\,.
 By \eqref{orainf3} and Lemma \ref{poincare} we have that $\tilde\beta_\ep\in\as_\ep(\tilde\mu_\ep)$ (for every $\ep>0$) and
 $$
 \F_\ep(\mu_\ep)\ge F_\ep(u_\ep,\sigma_\ep)-\ep^2\ge \frac{\sqrt{3}}{2} \ep^2E_\ep(\tilde\beta_\ep;A_\ep(\tilde \mu_\ep))-\ep^2\,,
 $$ 
 whence the claim follows by applying Theorem \ref{thmcont}(ii) with $\mu_\ep=\tilde\mu_\ep$ and $\beta_\ep=\tilde\beta_\ep$\,.
\end{proof}
\begin{proof}[Proof of Theorem \ref{mainthm}(iii)]
\begin{figure}
\definecolor{ffffff}{rgb}{1,1,1}
\definecolor{grigio}{rgb}{0.4,0.4,0.4}
\definecolor{uuuuuu}{rgb}{0,0,0}
\begin{tikzpicture}[line cap=round,line join=round,x=1cm,y=1cm,rotate=270,scale=1]
\clip(-.5,-1) rectangle (5.5,7);
\draw [line width=0.8pt] (0,0)-- (5.196152422706632,3);
\draw [line width=0.8pt] (0,1)-- (5.196152422706632,4);
\draw [line width=0.8pt] (0,2)-- (5.196152422706632,5);
\draw [line width=0.8pt] (0,3)-- (5.196152422706632,6);
\draw [line width=0.8pt] (0,4)-- (3.4641016151377544,6);
\draw [line width=0.8pt] (0,5)-- (1.7320508075688772,6);
\draw [line width=0.8pt] (1.7320508075688772,0)-- (5.196152422706632,2);
\draw [line width=0.8pt] (3.4641016151377544,0)-- (5.196152422706632,1);
\draw [line width=0.8pt] (5.196152422706632,0)-- (0,3);
\draw [line width=0.8pt] (5.196152422706632,1)-- (0,4);
\draw [line width=0.8pt] (5.196152422706632,2)-- (0,5);
\draw [line width=0.8pt] (5.196152422706632,3)-- (0,6);
\draw [line width=0.8pt] (5.196152422706632,4)-- (1.7320508075688772,6);
\draw [line width=0.8pt] (5.196152422706632,5)-- (3.4641016151377544,6);
\draw [line width=0.8pt] (3.4641016151377544,0)-- (0,2);
\draw [line width=0.8pt] (1.7320508075688772,0)-- (0,1);
\draw [line width=0.8pt] (0,0)-- (0,6);
\draw [line width=0.8pt] (0.8660254037844386,0.5)-- (0.8660254037844386,5.5);
\draw [line width=0.8pt] (1.7320508075688772,0)-- (1.7320508075688772,6);
\draw [line width=0.8pt] (2.598076211353316,0.5)-- (2.598076211353316,5.5);
\draw [line width=0.8pt] (3.4641016151377544,0)-- (3.4641016151377544,6);
\draw [line width=0.8pt] (4.330127018922193,0.5)-- (4.330127018922193,5.5);
\draw [line width=0.8pt] (5.196152422706632,0)-- (5.196152422706632,6);
\draw [line width=1.6pt] (0.57735026918962,1)-- (0.57735026918962,6); 
\draw [line width=1.6pt] (4.90747728811181893,4.5)-- (2.3094010767585029897,6); 
\draw [fill=grigio] (0,0) circle (2.5pt);
\draw [fill=grigio] (0,1) circle (2.5pt);
\draw [fill=grigio] (0,2) circle (2.5pt);
\draw [fill=grigio] (0,3) circle (2.5pt);
\draw [fill=grigio] (0,4) circle (2.5pt);
\draw [fill=grigio] (0,5) circle (2.5pt);
\draw [fill=grigio] (0,6) circle (2.5pt);
\draw [fill=grigio] (0.8660254037844386,0.5) circle (2.5pt);
\draw [fill=grigio] (0.8660254037844386,1.5) circle (2.5pt);
\draw [fill=grigio] (0.8660254037844386,2.5) circle (2.5pt);
\draw [fill=grigio] (0.8660254037844386,3.5) circle (2.5pt);
\draw [fill=grigio] (0.8660254037844386,4.5) circle (2.5pt);
\draw [fill=grigio] (0.8660254037844386,5.5) circle (2.5pt);
\draw [fill=grigio] (1.7320508075688772,0) circle (2.5pt);
\draw [fill=grigio] (1.7320508075688772,1) circle (2.5pt);
\draw [fill=grigio] (1.7320508075688772,2) circle (2.5pt);
\draw [fill=grigio] (1.7320508075688772,3) circle (2.5pt);
\draw [fill=grigio] (1.7320508075688772,4) circle (2.5pt);
\draw [fill=grigio] (1.7320508075688772,5) circle (2.5pt);
\draw [fill=grigio] (1.7320508075688772,6) circle (2.5pt);
\draw [fill=grigio] (2.598076211353316,0.5) circle (2.5pt);
\draw [fill=grigio] (2.598076211353316,1.5) circle (2.5pt);
\draw [fill=grigio] (2.598076211353316,2.5) circle (2.5pt);
\draw [fill=grigio] (2.598076211353316,3.5) circle (2.5pt);
\draw [fill=grigio] (2.598076211353316,4.5) circle (2.5pt);
\draw [fill=grigio] (2.598076211353316,5.5) circle (2.5pt);
\draw [fill=grigio] (3.4641016151377544,0) circle (2.5pt);
\draw [fill=grigio] (3.4641016151377544,1) circle (2.5pt);
\draw [fill=grigio] (3.4641016151377544,2) circle (2.5pt);
\draw [fill=grigio] (3.4641016151377544,3) circle (2.5pt);
\draw [fill=grigio] (3.4641016151377544,4) circle (2.5pt);
\draw [fill=grigio] (3.4641016151377544,5) circle (2.5pt);
\draw [fill=grigio] (3.4641016151377544,6) circle (2.5pt);
\draw [fill=grigio] (4.330127018922193,0.5) circle (2.5pt);
\draw [fill=grigio] (4.330127018922193,1.5) circle (2.5pt);
\draw [fill=grigio] (4.330127018922193,2.5) circle (2.5pt);
\draw [fill=grigio] (4.330127018922193,3.5) circle (2.5pt);
\draw [fill=grigio] (4.330127018922193,4.5) circle (2.5pt);
\draw [fill=grigio] (4.330127018922193,5.5) circle (2.5pt);
\draw [fill=grigio] (5.196152422706632,0) circle (2.5pt);
\draw [fill=grigio] (5.196152422706632,1) circle (2.5pt);
\draw [fill=grigio] (5.196152422706632,2) circle (2.5pt);
\draw [fill=grigio] (5.196152422706632,3) circle (2.5pt);
\draw [fill=grigio] (5.196152422706632,4) circle (2.5pt);
\draw [fill=grigio] (5.196152422706632,5) circle (2.5pt);
\draw [fill=grigio] (5.196152422706632,6) circle (2.5pt);
\draw [fill=uuuuuu] (0.5773502691896257,1) circle (2pt);
\draw[color=uuuuuu] (0.3,1.3) node {$x^{k_1\!,\ep}$};
\draw[color=uuuuuu] (0.65,6.4) node {$\gamma^{k_1\!,\ep}$};
\draw [fill=uuuuuu] (4.90747728811181893,4.5) circle (2pt); 
\draw[color=uuuuuu] (4.8,4.92) node {$x^{k_2\!,\ep}$};
\draw[color=uuuuuu] (2.37,6.4) node {$\gamma^{k_2\!,\ep}$};
\draw[color=uuuuuu] (-0.2,4.25) node {$i{+}\e\eta$};
\draw[color=uuuuuu] (1.2,3.5) node {$i$};
\draw[color=uuuuuu] (-0.2,2.85) node {$i{-}\e\eta$};
\end{tikzpicture}
\caption{Geometric construction used in the proof of Theorem \ref{mainthm}(iii). The position of each dislocation is approximated with a point $x^{k,\ep}$ sitting on the barycenter of a triangle of $\T_\ep(\Omega)$\,. 
The approximate displacement is the restriction to the lattice points of a continuum displacement having a jump on each half-line $\gamma^{k,\ep}$\,.
The sets $S^1_{k,\ep}$ defined in the proof contain bonds across the jumps: for example, the points $i$\,, $i{\pm}\ep\eta$ displayed in the picture are such that
both bonds $(i,i{-}\ep\eta)$\,, $(i,i{+}\ep\eta)$ belong to the set $S^1_{k_1,\ep}$\,.} \label{fig-proof}
\end{figure}
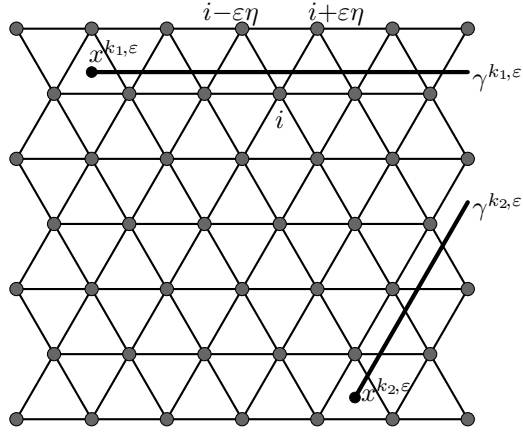
Let $\mu=\sum_{k=1}^{K}b^k\delta_{x^k}\in X(\Omega)$\,. 
By standard density arguments in $\Gamma$-convergence we can assume that $\ffi(b^k)=\psi(b^k)$ for every $k=1,\ldots,K$\,, i.e., that $|b^k|=1$ for every  $k=1,\ldots,K$\,.
Let $\ep\DT(\Omega)$ be the set of the barycenters of the triangles $T_\ep^+=i+\ep T^+\in\T_\ep(\Omega)$\,.

For every $k=1,\ldots,K$ and for every $\ep>0$\,, let $x^{k,\ep}\in\ep\DT(\Omega)$ be such that
$$
|x^{k,\ep}-x^k|=\min\{|y-x^k|\,:\,y\in \ep\DT(\Omega) \}\,.
$$ 
For every $k=1,\ldots,K$ set (see Figure \ref{fig-proof})
$$
\gamma^{k,\ep}:=\left\{
\begin{array}{ll}
\{x^{k,\ep}+\lambda e_1\,:\,\lambda\ge 0\}&\textrm{if }b^k =\pm\,  e_1\,,\\
\{x^{k,\ep}+\lambda \nu\,:\,\lambda\ge 0\}&\textrm{if }b^k=\pm\, \nu\,,\\
\{x^{k,\ep}+\lambda \eta\,:\,\lambda\ge 0\}&\textrm{if }b^k= \pm\, \eta\,.
\end{array}
\right.
$$
Let $\bar u^{k,\ep}\in C^2(\R^2\setminus \gamma^{k,\ep};\R^2)$ be a function satisfying the following properties,
\begin{equation}\label{bbb}
S_{\bar u^{k,\ep}}=\gamma^{k,\ep}\,,\qquad [\bar u^{k,\ep}]=b^k\textrm{ on }\gamma^{k,\ep}\,,\qquad \nabla \bar u^{k,\ep}(\cdot)=\beta^{b_k,\C}_{\R^2}(\cdot-x^{k,\ep})\,,
\end{equation}
where $\beta^{b,\C}_{\R^2}$ is defined in \eqref{betapiano}, for the choice of $g^{b,\C}$ and $f^{b,\C}(\theta)$ in \eqref{gfottime}.
Note that the function $\bar u^{k,\ep}$ is uniquely determined up to a constant.
We set $u^{k,\ep}:=\ep \bar u^{k,\ep}$ and  we define $u_\ep:\ep\Tl\to \R^2$ as 
$$
u_\ep(i):=\displaystyle \sum_{k=1}^{K}u^{k,\ep}(i)\,.
$$
Furthermore, for every $k=1,\ldots,K$ we set
\begin{equation*}
S^{1}_{k,\ep}:=\{(i,j)\in \ep\Tl\times \ep\Tl\,:\,|j-i|=\ep\,,\, [i,j]\cap\gamma^{k,\ep}\neq\emptyset\,,\,\di(i,\gamma^{k,\ep})<\di(j,\gamma^{k,\ep})\}\,,
\end{equation*}
where $[i,j]$ denotes the segment line with endpoints $i$ and $j$, and we define $\sigma_\ep:\ep\Tl\times\ep\Tl\to\ep\Tl$ as
$\sigma_\ep=\sum_{k=1}^{K}\sigma^{k,\ep}$\,, where
\begin{equation*}
\sigma^{k,\ep}(i,j):=
\begin{cases}
- \e b^k &\textrm{ if  }(i,j)\in S_{k,\ep}^1 \,,\\
 +\e b^k  &\textrm{ if }(j,i)\in S_{k,\ep}^1\,,\\
0&\textrm{ elsewhere}\,.
\end{cases}
\end{equation*}
Abusing notation, we still denote by  $u_\ep$ and $\sigma_\ep$ the restrictions of $u_\ep$ and $\sigma_\ep$ to $\Omega_\ep^0$ and $\Omega_\ep^1$, respectively. Then, $u_\ep\in\AD_\ep(\Omega)$ and $\sigma_\ep\in\AS_\ep(\Omega)$\,; 
moreover, for $\ep$ small enough,
\begin{equation*}
\mu_\ep:=\mu[\sigma_\ep]=\ep\sum_{k=1}^{K}b^k\delta_{{x^{k,\ep}}}\in X_\ep(\Omega),
\end{equation*}
and, by construction,
$\frac{\mu_\ep}{\ep}\weakstar\mu$ as $\ep\to 0$\,, which clearly implies $\frac{\mu_\ep}{\ep}\fla\mu$\,.

Therefore, in order to prove \eqref{form:limsup} it is enough to show that
\begin{equation}\label{newlimsup}
\limsup_{\ep\to 0}\frac{F_\ep(u_\ep, \sigma_\ep)}{\ep^2|\log\ep|}\le\frac{\sqrt{3}}{2}\sum_{k=1}^{K}\psi(b^k)\,.
\end{equation}
To this purpose, let $0<r<\frac{1}{4}\min\{\di_{k_1\neq k_2}(x^{k_1},x^{k_2})\,, \di(x^k,\partial\Omega)\}$\,.
In analogy with the notation introduced in Section \ref{sc:model}, for any open ball $B$
we denote by $B_\ep^1$ or $(B)_\ep^1$ the family of nearest neighbor bonds in $B$.
In order to show \eqref{newlimsup}, we preliminarily notice that there exists a constant $C>0$ depending only on $K$ such that
\begin{equation*}
\begin{aligned}
F_\e(u_\e,\sigma_\e)\leq \ &\frac{1}{2\ep^2}\sum_{k=1}^K\sum_{(i,j)\in (B_{2r}(x^{k,\ep}))_\ep^1}[(\ud u^{k,\ep}(i,j)-\sigma^{k,\ep}(i,j))\cdot(j-i)]^2\\
\ &+\frac{1}{2\ep^2}\sum_{k=1}^{K}\sum_{\newatop{l=1}{l\neq k}}^K\sum_{(i,j)\in (B_{2r}(x^{k,\ep}))_\ep^1} [(\ud u^{l,\ep}(i,j)-\sigma^{l,\ep}(i,j))\cdot(j-i)]^2\\
\ &+\frac{1}{\ep^2}\sum_{k=1}^{K}\sum_{\newatop{l=1}{l\neq k}}^K \sum_{(i,j)\in (B_{2r}(x^{k,\ep}))_\ep^1}\big[(\ud u^{k,\ep}(i,j)-\sigma^{k,\ep}(i,j))\cdot(j-i)\big]\\
\ &\phantom{+\frac{C}{\ep^2}\sum_{k=1}^{K}\sum_{\newatop{l=1}{l\neq k}}^K \sum_{(i,j)\in (B_{2r}(x^{k,\ep}))_\ep^1}}\qquad \big[(\ud u^{l,\ep}(i,j)-\sigma^{l,\ep}(i,j))\cdot(j-i)\big]\\
\ &+\frac{C}{\ep^2}\sum_{k=1}^K\sum_{(i,j)\in \Omega_\ep^1\setminus(\bigcup_{l=1}^K B_{r}(x^{l,\ep}))_\ep^1}[(\ud u^{k,\ep}(i,j)-\sigma^{k,\ep}(i,j))\cdot(j-i)]^2\,.
\end{aligned}
\end{equation*}
Therefore, \eqref{newlimsup} is proved once provided that for every fixed $k,l=1,\ldots,K$ with $l\neq k$ it holds 
\begin{eqnarray}\label{ls1}
&&\limsup_{\ep\to 0}\frac{1}{2\ep^4|\log\ep|}\sum_{(i,j)\in (B_{2r}(x^{k,\ep}))_\ep^1}[(\ud u^{k,\ep}(i,j)-\sigma^{k,\ep}(i,j))\cdot(j-i)]^2=\frac{\sqrt{3}}{2}\psi(b^k)\,,\\ \label{ls2}
&&\limsup_{\ep\to 0}\frac{1}{\ep^4|\log\ep|} \sum_{(i,j)\in (B_{2r}(x^{k,\ep}))_\ep^1} [(\ud u^{l,\ep}(i,j)-\sigma^{l,\ep}(i,j))\cdot(j-i)]^2=0\,,\\ \label{ls3}
&&\limsup_{\ep\to 0}\frac{1}{\ep^4|\log\ep|} \sum_{(i,j)\in (B_{2r}(x^{k,\ep}))_\ep^1}\big[(\ud u^{k,\ep}(i,j)-\sigma^{k,\ep}(i,j))\cdot(j-i)\big]\\ \nonumber
&&\phantom{+\frac{C}{\ep^2}\sum_{k=1}^{K}\sum_{\newatop{l=1}{l\neq k}}^K \sum_{(i,j)\in (B_{2r}(x^{k,\ep}))_\ep^1}}\qquad \big[(\ud u^{l,\ep}(i,j)-\sigma^{l,\ep}(i,j))\cdot(j-i)\big]=0\,,\\ \label{ls4}
&&\limsup_{\ep\to 0}\frac{1}{\ep^4|\log\ep|} \sum_{(i,j)\in \Omega_\ep^1\setminus(\bigcup_{l=1}^K B_{r}(x^{l,\ep}))_\ep^1}[(\ud u^{k,\ep}(i,j)-\sigma^{k,\ep}(i,j))\cdot(j-i)]^2=0\,.
\end{eqnarray}
Now, by the very definition of $\beta^{b,\C}_{\R^2}$  in \eqref{betapiano}, for every $b\in\Tl$\,, we have that there exists a universal constant $C>0$ such that
\begin{equation}\label{decay}
|\beta^{b,\C}_{\R^2}(\rho,\theta)|\leq \frac{C}{\rho},\qquad\qquad |\nabla \beta_{\R^2}^{b,\C}(\rho,\theta)|\leq \frac{C}{\rho^2}\,;
\end{equation}
moreover, by the very definition of $u^{k,\ep}$ and $\sigma^{k,\ep}$\,, 
\begin{equation}\label{intdir}
\begin{aligned}
\frac {\ud u^{k,\e}(i,j)}{\e}-\frac{\sigma^{k,\e}(i,j)}{\e}= \int_0^1 \beta^{b^k,\C}_{\R^2}(i+t(j-i)-x^{k,\ep})\,(j-i)\,\ud t\,,
\end{aligned}
\end{equation}
for every $k=1,\ldots,K$\,, for every $T_\ep\in\T_\ep$\,, and for every $(i,j)\in(T_\ep)_\ep^1$\,.
As a consequence, for every $k=1,\ldots,K$\,, for every $T_\ep\in{\T_\ep}$ with $\mu_\ep(T_\ep)=0$\,, for every $(i,j)\in(T_\ep)_\e^1$, and for every $x\in T_\ep$\,, we get
\begin{equation}\label{stima1}
\begin{aligned}
&\ \frac{1}{\ep^2}\Big[\Big({\ud u^{k,\e}(i,j)}-{\sigma^{k,\e}(i,j)}\Big)\cdot(j-i)\Big]^2
=\ \Big[\Big(\frac {\ud u^{k,\e}(i,j)}{\e}-\frac{\sigma^{k,\e}(i,j)}{\e}\Big)\cdot(j-i)\Big]^2\\
=\ &\e^4\Bigg[ \int_0^1 \Big(\frac{j-i}{\e}\Big)^*\beta^{b^k,\C}_{\R^2}(i+t(j-i)-x^{k,\ep})\Big(\frac{j-i}{\e}\Big)\,\ud t\Bigg]^2\\
\leq\ & \e^4 \int_0^1 \Big[\Big(\frac{j-i}{\e}\Big)^*\beta^{b^k,\C}_{\R^2}(i+t(j-i)-x^{k,\ep})\Big(\frac{j-i}{\e}\Big)\Big]^2\,\ud t\\
=\ &\e^4\Big[\Big(\frac{j-i}{\e}\Big)^*\beta^{b^k,\C}_{\R^2}(x-x^{k,\ep})\Big(\frac{j-i}{\e}\Big)\Big]^2\\
&\ +\ep^4\int_0^1\bigg( \Big[\Big(\frac{j-i}{\e}\Big)^*\beta^{b^k,\C}_{\R^2}(i+t(j-i)-x^{k,\ep})\Big(\frac{j-i}{\e}\Big)\Big]^2\\
&\ \phantom{\int_0^1\Big( \Big[\Big(\frac{j-i}{\e}\Big)^*\beta^{b^k,\C}_{\R^2}(i+t(j-i)-x^{k,\ep})}-\Big[\Big(\frac{j-i}{\e}\Big)^*\beta^{b^k,\C}_{\R^2}(x-x^{k,\ep})\Big(\frac{j-i}{\e}\Big)\Big]^2\Big)\, \ud t\bigg)\,,
\end{aligned}
\end{equation}
where the second equality follows from \eqref{intdir},  and the inequality is a consequence of Jensen's inequality. 
Now, by \eqref{decay} for every $k=1,\ldots,K$ and for every $\xi\in\{\pm e_1,\pm\nu,\pm\eta\}$ it holds
$$
\big| \nabla\big( [\xi^*\beta^{b^k,\C}_{\R^2}\xi]^2\big) (x-x^{k,\ep})\big|\leq \frac{C}{|x-x^{k,\ep}|^3}\,\qquad\textrm{ for every }x\neq x^{k,\ep}\,.
$$
In particular, for every $k=1,\ldots,K$\,, for every $T_\ep\in{\T_\ep}$ with $\mu_\ep(T_\ep)=0$\,, for every $(i,j)\in(T_\ep)_\e^1$, and for every $x\in T_\ep$\,,
$$
\Big[\Big(\frac{j-i}{\e}\Big)^*\beta^{b^k,\C}_{\R^2}(i+t(j-i)-x^{k,\ep})\Big(\frac{j-i}{\e}\Big)\Big]^2
-\Big[\Big(\frac{j-i}{\e}\Big)^*\beta^{b^k,\C}_{\R^2}(x-x^{k,\ep})\Big(\frac{j-i}{\e}\Big)\Big]^2
\leq \frac{C\,\ep}{|x-x^{k,\ep}|^3} \,.
$$
Hence, for every $T_\ep\in\T_\ep$ with $\mu_\ep(T_\ep)=0$\,,
by \eqref{stima1} and integrating over $T_\ep$\,, we get
\begin{equation}\label{qq}
\begin{aligned}
&F_\ep(u^{k,\ep},\sigma^{k,\ep};T_\ep)\\
\le\ &  \frac{4}{\sqrt{3}}\ep^2\int_{T_\ep}\Big(\big|e_1^*\beta^{b^k,\C}_{\R^2}(x-x^{k,\ep})e_1\big|^2
+\big|\nu^*\beta^{b^k,\C}_{\R^2}(x-x^{k,\ep})\nu\big|^2+\big|\eta^*\beta^{b^k,\C}_{\R^2}(x-x^{k,\ep})\eta\big|^2\Big)\ud x\\
\ & +C\e^3\int_{T_\ep}\frac{1}{|x-x^{k,\ep}|^3}\,\ud x\\
=\ &\sqrt 3\ep^2\int_{T_\ep}\frac{1}{2}\C\beta_{\R^2}^{b^k,\C}(x-x^{k,\ep}):\beta_{\R^2}^{b^k,\C}(x-x^{k,\ep})\,\ud x+C\e^3\int_{T_\ep}\frac{1}{|x-x^{k,\ep}|^3}\,\ud x\,,
\end{aligned}
\end{equation}
where the equality follows by the second equality in \eqref{ok1} with $\beta$ replaced by $\beta_{\R^2}^{b^k,\C}(x-x^{k,\ep})$\,.
Now, denoting by $T_\ep^{x^{k,\ep}}\in\T_\ep(\Omega)$ the triangle centered at $x^{k,\ep}$\,,
by \eqref{decay} and \eqref{intdir} for every $k=1,\ldots,K$ we get
\begin{equation*}
F_\ep(u^{k,\ep},\sigma^{k,\ep};T_\ep^{x^{k,\ep}})\le C\ep^2\,,
\end{equation*}
for some constant $C$ independent of $\ep$\,.
Moreover, recalling the notation previously introduced for the annuli,
an integration in polar coordinates shows that
$$
\int_{A_{\delta, 2r}(x^{k,\ep})} \frac{1}{|x-x^{k,\ep}|^3}\,\ud x \leq \frac{C}\delta \,,
$$
for every $0<\delta<2r$\,.
Therefore, by \eqref{qq} and \eqref{servealimsup}\,, for every $k=1,\ldots,K$ and for $c>0$ small enough (as for instance $c<\frac{\sqrt{3}}{6}$) we obtain
\begin{equation}\label{termine1}
\begin{aligned}
\ &\frac{1}{2\ep^2}\sum_{(i,j)\in (B_{2r}(x^{k,\ep}))_\ep^1}\Big[\Big({\ud u^{k,\e}(i,j)}-{\sigma^{k,\e}(i,j)}\Big)\cdot(j-i)\Big]^2\\
\le\ &\frac{1}{2}\sum_{T_\ep\in\T_\ep(B_{2r}(x^{k,\ep}))}F_\ep(u^{k,\ep},\sigma^{k,\ep};T_\ep)\\
\le\ &\frac{\sqrt{3}}{2}\ep^2\int_{A_{c\ep, 2r}(x^{k,\ep})}\frac 1 2\C\beta^{b^k,\C}_{\R^2}(x-x^{k,\ep}):\beta^{b^k,\C}_{\R^2}(x-x^{k,\ep})\,\ud x+C\ep^2\\
\le\ &\ep^2|\log\ep|\frac{\sqrt{3}}{2}\psi(b^k)+C\ep^2\,,
\end{aligned}
\end{equation}
where the constant $C$ changes from line to line; this proves \eqref{ls1}\,.

Moreover, by \eqref{qq} together with \eqref{decay}, for every $k,l=1,\ldots,K$ with $l\neq k$\,, we have 
\begin{equation}\label{termine2}
\begin{aligned}
&\frac{1}{\ep^4|\log\ep|} \sum_{(i,j)\in (B_{2r}(x^{k,\ep}))_\ep^1} \big[(\ud u^{l,\ep}(i,j)-\sigma^{l,\ep}(i,j))\cdot(j-i)\big]^2\\
\le\ &
\frac{1}{\ep^2|\log\ep|} \sum_{T_\ep\in \T_\ep(B_{2r}(x^{k,\ep}))}
F_\ep(u^{l,\ep},\sigma^{l,\ep}; T_\ep)\\
\le\ &\frac{C}{|\log\ep|}\Big(\int_{B_{2r}(x^{k,\ep})}\C\beta^{b^l,\C}_{\R^2}(x-x^{l,\ep}):\beta^{b^l,\C}_{\R^2}(x-x^{l,\ep})\,\ud x +\ep\int_{B_{2r}(x^{k,\ep})}\frac{1}{|x-x^{l,\ep}|^3}\,\ud x\Big)\\
\le\ & \frac{C_r}{|\log\ep|}\to 0
\quad\textrm{ as }\ep\to 0\,,
\end{aligned}
\end{equation}
where $C_r$ depends only on $r$\,.
Now, by \eqref{termine2} we obtain \eqref{ls2}\,.
Analogously, one can prove that also \eqref{ls4} holds true.

Finally, by the H\"older inequality, using \eqref{termine1} and \eqref{termine2}, for every $k,l=1,\ldots,K$ with $l\neq k$ we get
\begin{equation*}
\begin{aligned}
&\frac{1}{\ep^4|\log\ep|}\sum_{(i,j)\in (B_{2r}(x^{k,\ep}))_\ep^1}\big[(\ud u^{k,\ep}(i,j)-\sigma^{k,\ep}(i,j))\cdot(j-i)\big]
\big[(\ud u^{l,\ep}(i,j)-\sigma^{l,\ep}(i,j))\cdot(j-i)\big]\\
\le\ &\bigg(\frac{1}{\ep^4|\log\ep|}\sum_{(i,j)\in (B_{2r}(x^{k,\ep}))_\ep^1}\big[(\ud u^{k,\ep}(i,j)-\sigma^{k,\ep}(i,j))\cdot(j-i)\big]^2\bigg)^{\frac{1}{2}}\\
&\phantom{\frac{1}{\ep^4|\log\ep|}\bigg(\sum_{(i,j)\in (B_{2r}(x^{k,\ep}))_\ep^1}}\bigg(\frac{1}{\ep^4|\log\ep|}\sum_{(i,j)\in (B_{2r}(x^{k,\ep}))_\ep^1}\big[(\ud u^{l,\ep}(i,j)-\sigma^{l,\ep}(i,j))\cdot(j-i)\big]^2\bigg)^{\frac{1}{2}}\\
\le\ & \frac{C}{\sqrt{|\log\ep|}}\to 0\qquad\textrm{as }\ep\to 0\,,
\end{aligned}
\end{equation*}
which proves \eqref{ls3} and concludes the proof of the $\Gamma$-limsup inequality. 
\end{proof}
\section{Degeneracies of the model and possible kinematic constraints}\label{finalcomment}
We highlight that the functional \eqref{def:en} exhibits some degeneracies that are typical of linearized energies as well as of discrete models built on a reference configuration; the presence of the plastic slip field $\sigma$ actually makes the system even less rigid and produces even more unphysical configurations. 

As an example, let $\Omega:=[-1,1]^2$ and define $u_\ep^\pm\in\AD_\ep(\Omega)$ as
\begin{equation}\label{crack}
u_\ep^\pm(i):=\left\{
\begin{array}{ll}
\pm\ep\nu&\textrm{if }i_2\le 0\\
0&\textrm{elsewhere in }\Omega_\ep^0\,.
\end{array}
\right.
\end{equation}
Notice that $u_\ep^{-}$ can be interpreted as a microscopic crack, whereas $u_\ep^+$ can be seen as an unphysical crack which produces interpenetration of matter
generating the superposition of two lines of atoms.
It is easy to check that there exist $\sigma_\ep^{\pm}\in\AS_\ep(\Omega)$ such that $F_\ep(u_\ep^\pm,\sigma_\ep^\pm)=0$\,. 
Moreover, up to slightly modifying $u_\ep^-$, one can construct a function $u_\ep\in\AD_\ep(\Omega)$ such that the crack does not disconnect $\Omega$ and has two ending points which in our model are classified as dislocations with opposite Burgers vector. In particular, what in our model is identified as a plastic slip, 
in some cases should be rather understood as a microscopic crack opening.

Furthermore, we observe that the energy in \eqref{def:en} is invariant with respect  to integer dilations of the lattice. In order to see this, it is enough to fix $\lambda\in\Z$ and to take the displacement $u_\ep^\lambda$ and the slip $\sigma_\ep^\lambda$ defined by 
\begin{equation}\label{omo}
u_{\ep}^{\lambda}(i):=\lambda\ep i\,\qquad\qquad\sigma_\ep^\lambda(i,j):=\ep\lambda(j{-}i)\,.
\end{equation}
Trivially, $\ud u_\ep^{\lambda}(i,j)\equiv \sigma_\ep^\lambda(i,j)$ and hence $F_\ep(u_\ep^\lambda,\sigma_\ep^\lambda)\equiv 0$\,.

In view of the degeneracies pointed out above, one may ask whether additional suitable kinematic constraints on the plastic 
slips could prevent such unphysical  behaviors.
In the present  framework, it seems  natural to incorporate linearized constraints on the plastic slips, mimicking deviatoric/pure shear  volume preserving conditions on each $\ep$-triangle $T_\ep\in\T_\ep$\,.

Let us discuss now a possible way to  introduce such a kinematic constraint in our model. 
In order to 
  (formally) linearize the nonlinear volume preserving constraint,  consider 
deformation gradients  $\mathrm{Id}+\delta z$\, close to the identity, where  $z\in\R^{2\times 2}$\, represents a plastic deformation gradient,  $\delta>0$ is a small parameter with respect to which the linearization is performed and $\mathrm{Id}$ is the identity matrix in $\R^{2\times 2}$\,. 
Enforcing that the deformed configuration $(\mathrm{Id}+\delta z) (T_\ep)$ has the same (oriented)  area of $T_\e$ we obtain
\[
[(k{-}i)+\delta\,z(k{-}i)]\wedge[(j{-}i)+\delta\,z(j{-}i)] =(k{-}i)\wedge(j{-}i) \,,
\]
for every triple $(i,j,k)$ of counterclockwise oriented vertices of $T_\ep$\,.
Neglecting lower order terms, this yields 
\begin{equation}\label{condliz}
z(k{-}i)\wedge (j{-}i)-z(j{-}i)\wedge (k{-}i)=0\,,
\end{equation} 
for every triple $(i,j,k)$ as above.
Summing \eqref{condliz} over two triples of the type $(i,j,k)$, $(j,k,i)$, one can see that condition \eqref{condliz} is equivalent to the well known  trace free constraint on $z$ ;
we refer the interested reader to  \cite{MS} for a rigorous derivation via $\Gamma$-convergence of such a kind of trace free constraints  starting from continuous nonlinear models in plasticity, relying on  multiplicative rather than additive decompositions of the deformation gradient. 

We now rewrite \eqref{condliz} in terms of the (linear) slip field $\sigma$\,, which is, by construction, given by
\begin{equation}\label{legamesigmaz}
\sigma(i,j)=z(j{-}i)\qquad\qquad\textrm{for every bond }(i,j)\textrm{ of the triangle }T_\ep\,.
\end{equation}
Condition \eqref{condliz} becomes
\begin{equation}\label{condli}
\sigma(i,k)\wedge (j{-}i)-\sigma(i,j)\wedge (k{-}i)=0\,,
\end{equation} 
for every triple $(i,j,k)$ of counterclockwise oriented vertices of $T_\ep$\,.
As a consequence of \eqref{legamesigmaz},  we have that $\mu[\sigma](T_\ep)=0$\,, and hence \eqref{condli} is justified by the reasoning above only on dislocation-free triangles.
On the other hand, one can easily check that the only condition \eqref{condli} (assumed for every triple $(i,j,k)$ of counterclockwise oriented vertices of $T_\ep$) implies that $\mu[\sigma](T_\ep)=0$\,.
Summarizing, in our framework, a reasonable linearization of the volume preserving constraint seems to be provided by condition \eqref{condli} assumed  on all dislocation-free $\ep$-triangles $T_\ep$\,.

Notice that \eqref{condli} rules out the degeneracies pointed out by the examples above. More precisely, one can check that, given a displacement as in \eqref{crack} or\eqref{omo}, the minimum energy  in a dislocation-free triangle among all slip fields fulfilling \eqref{condli} is strictly positive. On the other hand  \eqref{MSass} guarantees that dislocation-free triangles are necessarily present. Finally, our $\Gamma$-convergence result still holds true under the additional constraint \eqref{condli}, 
which is satisfied by the recovery sequence constructed in the proof of Theorem \ref{mainthm}(iii).

\end{document}